\documentclass[11pt,draft]{elsarticle}
\usepackage{amssymb,amsmath}
\usepackage{epic}
\usepackage{pstricks}
\usepackage{color}
\usepackage[ansinew]{inputenc}
\usepackage[english]{babel}
\usepackage{latexsym}
\usepackage{eufrak}
\usepackage{mathrsfs}
\usepackage{fancybox}
\usepackage{fancyhdr}
\usepackage{array}
\newtheorem{thm}{Theorem}[section]
\newtheorem{cor}[thm]{Corollary}

\newtheorem{prop}[thm]{Proposition}

\newtheorem{rem}{Remark}[section]

\newenvironment{proof}{{\bf Proof:\ }}
{\hfill $\Box$}
\newcommand\XX[1]{\mathbb{#1}}
\begin{document}
\begin{frontmatter}
\title{Zeros of polynomials orthogonal with respect to a signed weight}
\author[BE]{M. Benabdallah}
\ead{benabdallahmajed@gmail.com}
\author[AT]{M. J. Atia \corref{cor1}}
\ead{jalel.atia@gmail.com}
\author[RCS]{R. S. Costas-Santos \fnref{fn1}}
\ead{rscosa@gmail.com}
\ead[url]{http://www.rscosan.com}
\cortext[cor1]{Corresponding author}

\fntext[fn1]{RSCS acknowledges financial
support from the Ministerio de Ciencia e
Innovaci\'on of Spain and under the
program of postdoctoral grants (Programa
de becas postdoctorales) and grant
MTM2009-12740-C03-01.}

\address[AT]{Facult\'e sciences Gabes, cit\'e Erriadh, 6072 Gabes Tunisie.}
\address[BE]{Lyc\'ee Blidet, Blidet Kebili, Tunise}
\address[RCS]{Department of Mathematics, University of
California, South Hall, Room 6607 Santa Barbara, CA 93106 USA}
\begin{abstract}
\noindent
In this paper we consider the polynomial
sequence $(P_{n}^{\alpha,q}(x))$ that is
orthogonal on $[-1,1]$ with respect to the
weight function $x^{2q+1}(1-x^{2})^{\alpha}(1-x),
\alpha>-1, q\in \XX N$; we obtain the
coefficients of the tree-term recurrence
relation (TTRR) by using a different method
from the one derived in \cite{kn:atia1};
we prove that the interlacing property does
not hold properly for $(P_n^{\alpha,q}(x))$;
and we also prove that, if $x_{n,n}^{\alpha+i,q+j}$
is the largest zero of $P_{n}^{\alpha+i,q+j}(x)$,
$\displaystyle x_{2n-2j,2n-2j}^{\alpha+j,q+j}<
x_{2n-2i,2n-2i}^{\alpha+i,q+i}, 0\leq i<j\leq n-1$.
\end{abstract}

\end{frontmatter}

\section{Introduction}
It is a well known fact that if $(p_n)$ is
orthogonal with respect to a (real) weight
function, namely $w(x)$, and such weight function
is positive on $[a,b]$, then the zeros of $p_n$ are
real, distinct, interlace, and lie inside $]a,b[$,
but such interlacing property is no longer valid
when the weight is a signed function.
In fact, Perron \cite{kn:perron} proved that when
the $w(x)$ changes sign once then one of zeros can
lie outside of $[a,b]$.

In this paper we prove that such zero can lie into
one of the endpoints of the interval, $a$ or $b$.
We consider the weight function $w(x)=x^{2q+1}
(1-x^{2})^{\alpha}(1-x)$, $\alpha>-1$, $q\in \XX N$
that changes sign once, at $x=0$, and we prove that
all the zeros are real, non interlacing, and that
one of the zeros is the endpoint $a=-1$.

A sequence of monic orthogonal polynomials
$(\hat Q_{n})$ satisfies for $n\ge 0$ the
following TTRR \cite{kn:chihara}:
\begin{equation}\label{TTRR}
\hat Q_{n+2}(x)=(x-\beta_{n+1})\hat Q_{n+1}(x)
-\gamma_{n+1}\hat Q_{n}(x),
\end{equation}
with initial conditions $\hat Q_{0}(x)=1$, $\hat
Q_{1}(x)=x-\beta_0$, being $(\beta_n)$ and $(\gamma_n)$ the
coefficients of the recurrence relation. The recurrence
coefficients of $(P_{n}^{\alpha,q})$ were calculated in
\cite{kn:atia1} by using the Laguerre-Freud equations, and, later
on, an explicit expression for $P_{n}^{\alpha,q}(x)$ was given in
\cite{kn:atia2}. The main aim of this paper is to keep studying
these polynomials, more precisely, the behavior of the zeros of
$(P_{n}^{\alpha,q}(x))$.

People working on zeros of orthogonal polynomials
know how difficult is to explore this area.
In fact, even in the case of Jacobi polynomials
results on zeros are presented as conjectures
(see \cite{kn:gautshi1}, \cite{kn:gautshi2}).

In order to do this study, we use generalized
Gegenbauer polynomials $(GG_{n}^{\alpha,\mu})$
that are orthogonal on $[-1,1]$ with respect to
the weight function $|x|^\mu(1-x^{2})^{\alpha}$,
$\alpha>-1$, $\mu>-1$.
Some properties of GG-polynomials can be found
in \cite{kn:szego} and \cite{kn:atia5}.

The structure of the paper is the following:
in Section 2 we present basic definitions, some
notation, and a few preliminary results, in
Section 3 we obtain some algebraic relations
between $(P_n^{\alpha,q}(x))$ and the
GG-polynomials as well as the recurrence
coefficients of the TTRR fulfilled by
$(P_n^{\alpha,q}(x))$, and in Section 4 some
results regarding zeros of $(P_n^{\alpha,q}(x))$
are given.
\section{Basic definitions and
preliminary results}
The \textit{Pochhammer symbol}, or shifted
factorial, is defined as
\begin{equation}\label{pochammer}
(\alpha)_{0}=1, \quad (\alpha)_{n}=\alpha(\alpha+1)
\cdots(\alpha+n-1),\qquad n\geq1.
\end{equation}
The \textit{Gauss's hypergeometric function} is
\begin{equation}\label{fnhyperg}
_2F_{1}(\alpha,\beta;\gamma;z)=\sum_{n=0}^{\infty}
\frac{(\alpha)_{n}(\beta)_{n}}{n!(\gamma)_{n}}
z^{n}, \quad \alpha, \beta\in\mathbb{C};\gamma\in
\mathbb{C}\backslash\mathbb{Z}_-; |z|<1.
\end{equation}
When $\alpha$, or $\beta$, is a negative integer
the hypergeometric serie \eqref{fnhyperg}
terminates, i.e. it reduces to a polynomial (of
degree $\alpha$, or $\beta$ resp.).

Observe that after straightforward calculation one
gets
\begin{equation}\label{fnhypergpolyno}
_{2}F_{1}(-m,\beta;\gamma;z)=\sum_{n=0}^{m}
\frac{(-m)_{n}(\beta)_{n}}
{n!(\gamma)_{n}}z^{n}=\sum_{n=0}^{m}
\binom m n\frac{(\beta)_{n}}
{(\gamma)_{n}}(-z)^{n}.
\end{equation}
Denoting by
$$
_{2}F_{1}(\alpha,\beta;\gamma;z)\equiv F,
\quad _{2}F_{1}(\alpha\pm1,\beta;\gamma;z)
\equiv F(\alpha\pm1),
$$
$$
_{2}F_{1}(\alpha,\beta\pm1;\gamma;z)\equiv
F(\beta\pm1),\quad _{2}F_{1}(\alpha,\beta;
\gamma\pm;z)\equiv F(\gamma\pm1).
$$
The functions $F(\alpha\pm1)$, $F(\beta\pm1)$,
and $F(\gamma\pm1)$ are said to be {\it contiguous}
of $F$ \cite[p. 242]{kn:lebedev}.
Among the relations of this type we cite the
following ones:
\begin{equation}\label{contig1}
(\gamma-\alpha-\beta)F+\alpha(1-z)F(\alpha+1)
-(\gamma-\beta)F(\beta-1)=0,
\end{equation}
\begin{equation}\label{contig2}
(\gamma-\alpha-1)F+\alpha F(\alpha+1)-(\gamma-1)
F(\gamma-1)=0,
\end{equation}
\begin{equation}\label{contig3}
\gamma(1-z)F-\gamma F(\alpha-1)+(\gamma-\beta)z
F(\gamma+1)=0,
\end{equation}
\begin{equation}\label{contig4}
(\alpha-\beta)F-\alpha F(\alpha+1)+\beta
F(\beta+1)=0,
\end{equation}
\begin{equation}\label{contig5}
(\alpha-\beta)(1-z)F+(\gamma-\alpha)F(\alpha-1)
-(\gamma-\beta)F(\beta-1)=0.
\end{equation}
The following result allow us to find new
relations between the the polynomial sequence
$(P_n^{\alpha,q}(x))$ and the GG-polynomials:
\begin{prop} \label{prop2.1}
GG polynomials are
related by the following relation
\begin{equation}\label{relationpmu}
GG_{2n+1}^{\alpha,\mu}(x)=xGG_{2n}^{\alpha,
\mu+2}(x),\ n\geq0.
\end{equation}
\end{prop}
\begin{proof} This comes directly from
\begin{equation}\label{gg2n}
GG_{2n}^{\alpha,\mu}(x)=x^{2n}\ _{2}F_{1}
\Big(-n,-n-\frac{\mu}{2}
+\frac{1}{2};-2n-\alpha-\frac{\mu}{2}
+\frac{1}{2};\frac{1}{x^{2}}\Big),\ n\geq0,
\end{equation}
and
\begin{equation}\label{gg2np1}
GG_{2n+1}^{\alpha,\mu}(x)=x^{2n+1}\
_{2}F_{1}\Big(-n,-n-\frac{\mu}{2}
-\frac{1}{2};-2n-\alpha-\frac{\mu}{2}
-\frac{1}{2};\frac{1}{x^{2}}\Big),\ n\geq0.
\end{equation}
\end{proof}

The following result characterizes
the polynomial sequence, up to a constant,
through its property of orthogonality:
\begin{thm}\cite{kn:lebedev} \label{theo2.2}
Let $(p_n)$ be a sequence of polynomials.
The following statements are equivalent:
\par $(a)$\ $(p_n)$ is a orthogonal
polynomial sequence with respect to the weight
$w(x)$ on $[a,b]$.
\par$(b)$ $\displaystyle
\int_{a}^{b}w(x) \pi(x) p_{n}(x)dx=0$ for
every polynomial $\pi$ of degree $m<n$ and
$\displaystyle\int_{a}^{b}w(x) \pi(x) p_{n}(x)dx
\neq0$ if $m=n$.
\par$(c)$ $\displaystyle\int_{a}^{b}w(x) x^{m}
p_{n}(x)dx=k_{n}\delta_{m,n}$ with $k_{n}\neq0$,
$0\leq m\leq n$.
\end{thm}
\begin{cor}
Let $(p_{n})$ be a orthogonal polynomial sequence
with respect to the weight $w(x)$ on $[a,b]$ and
let $(Q_{n})$ be another polynomial
sequence that fulfills the following property
of orthogonality:
\begin{equation}\label{ortho1}
\left\{
  \begin{array}{ll}
  \displaystyle  \int_{a}^{b}w(x)x^{k} Q_{n+r}(x)dx=0,
 & 0\leq k\leq n-1, \\[3mm]
   \displaystyle \int_{a}^{b}w(x)x^{n}
 Q_{n+r}(x)dx\neq 0, & n\geq 0,
  \end{array}
\right.
\end{equation}
then
\begin{equation}\label{ortho2}
Q_{n+r}(x)=\sum_{i=n}^{n+r}\lambda_{i}p_{i}(x),
\quad(\lambda_{i})\in\mathbb{C},\ n,r\in\mathbb{N}.
\end{equation}
\end{cor}
\section{Algebraic relations between $(P_n^{\alpha,
q}(x))$ and the GG-po\-ly\-no\-mials}
\begin{prop}
For any $n\ge 0$, and any integer $q$ the
following identities hold:
\begin{equation}\label{p2ngg2n}
P_{2n}^{\alpha,q}(x)=GG_{2n}^{\alpha,2q+2}(x),
\end{equation}

\begin{equation}\label{p2np1gg2n}
P_{2n+1}^{\alpha,q}(x)=(1+x) GG_{2n}^{\alpha+1,
2q+2}(x).
\end{equation}
\end{prop}

\begin{rem}
Observe that with \eqref{p2ngg2n}, we can write
the last equation as
\begin{equation}\label{p2np1gg2n2}
P_{2n+1}^{\alpha,q}(x)=(1+x) P_{2n}^{\alpha+1,
q}(x),
\end{equation}
one should point out that we have $\alpha$
in the left hand side whereas we have $\alpha+1$
in the right hand side.
\end{rem}

\begin{proof}
Let us start with the first identity.
For $n\ge 1$, we get
$$
\int_{-1}^{1}|x|^{\mu}(1-x^2)^{\alpha}x^{k}
GG_{2n}^{\alpha,\mu}(x)dx=0,\
0\leq k\leq 2n-1,
$$
taking $\mu=2q+2$, for $1\leq k\leq 2n-1$ we get
$$
\int_{-1}^{1}x^{2q+1}(1-x^{2})^{\alpha}(1-x)x^{k}
GG_{2n}^{\alpha,2q+2}(x)dx=0.
$$
Moreover, if $k=0$ a direct calculation shows
$$
\int_{-1}^{1}x^{2q+1}(1-x^{2})
^{\alpha}(1-x)GG_{2n}^{\alpha,2q+2}(x)dx=0,
$$
since it can be split in two parts and by using
Proposition \ref{prop2.1}, and the property of
orthogonality of $(GG_{n}^{\alpha,\mu})$ the
previous integral vanishes.

By analogous reasons we also obtain
\begin{equation*} \begin{split}
\int_{-1}^{1}x^{2q+1}(1-x^{2})
^{\alpha}(1-x)& x^{2n}GG_{2n}^{\alpha,2q+2}(x)dx
=\\ & -\int_{-1}^{1}x^{2q+2}(1-x^{2})
^{\alpha}x^{2n}GG_{2n}^{\alpha,2q+2}(x)dx\neq0.
\end{split} \end{equation*}
Hence, by Theorem \ref{theo2.2}, the first
identity holds.
\begin{rem} By using Eq. \eqref{gg2n} we get for
$n\ge 0$ the hypergeometric representation for
$P_{2n}^{\alpha,q}(x)$:
\begin{equation}\label{p2nhyper1}
P_{2n}^{\alpha,q}(x)=x^{2n}\ _{2}F_{1}\Big(-n,-n-
q-\frac{1}{2};-2n-\alpha-q-\frac{1}{2};\frac{1}
{x^{2}}\Big),
\end{equation}
that we can be written as
\cite[V1 p. 40 (23)]{kn:yudell}
\begin{equation}\label{p2nhyper2}
P_{2n}^{\alpha,q}(x)=\frac{(-1)^{n}(q+\frac{3}
{2})_{n}}{(n+q+\alpha+\frac{3}{2})_{n}}\ _{2}F_{1}
\Big(-n,n+q+\alpha+\frac{3}{2};q+\frac{3}{2};x^{2}
\Big).
\end{equation}
\end{rem}
Next let us prove the second identity.
For $n\ge 1$, we know
$$
\int_{-1}^{1}|x|^{\mu}(1-x^2)^{\alpha}x^{k}
GG_{2n}^{\alpha,\mu}(x)dx=0,
$$
$$
\int_{-1}^{1}|x|^{\mu}(1-x^2)^{\alpha} x^{2n}
GG_{2n}^{\alpha,\mu}(x)dx\neq0.
$$
So, setting $\alpha\leftarrow \alpha+1$,
$\mu\leftarrow 2q+2$, $0\leq k\leq 2n-1$, we get
$$
\int_{-1}^{1} x^{2q+2}(1-x^2)^{\alpha+1}
x^{k}GG_{2n}^{\alpha+1,2q+2}(x)dx=0,
$$
$$
\int_{-1}^{1}x^{2q+1}(1-x^2)^{\alpha}(1-x)x^{k+1}
\Big((1+x)GG_{2n}^{\alpha+1,2q+2}(x)\Big)dx=0.
$$
So, for $1\leq k\leq 2n$, we have
$$
\int_{-1}^{1} x^{2q+1}(1-x^2)^{\alpha}(1-x)x^{k}
\Big((1+x)GG_{2n}^{\alpha+1,2q+2}(x)\Big)dx=0,
$$
and if $k=0$, by parity of $GG_{2n}^{\alpha+1,
2q+2}(x)$, we obtain
$$
\int_{-1}^{1} x^{2q+1}(1-x^2)^{\alpha}(1-x)
\Big((1+x)GG_{2n}^{\alpha+1,2q+2}(x)\Big)dx=0.
$$

Therefore, since $\deg((1+x)GG_{2n}^{\alpha+1,
2q+2}(x))=2n+1$ and
\begin{equation*}\begin{split}
\int_{-1}^{1}x^{2q+1}(1-x^2)^{\alpha}(1-x)&x^{2n+1}
\Big((1+x)GG_{2n}^{\alpha+1,2q+2}(x)\Big)dx=
\\ & \int_{-1}^{1} x^{2q+2}(1-x^2)^{\alpha+1}
x^{2n}GG_{2n}^{\alpha+1,2q+2}(x)dx\neq0.
\end{split}\end{equation*}
Then, by Theorem \ref{theo2.2}, we have
\begin{equation*}
(1+x)GG_{2n}^{\alpha+1,2q+2}(x)=P_{2n+1}^{\alpha,
q}(x),
\end{equation*}
and hence the second identity holds.
\end{proof}
Once we have got these algebraic relations
we can compute the recurrences coefficients
associated to the polynomial sequence
$(P_{n}^{\alpha,q}(x))$.
\begin{rem}
Notice that due the expression of the integrals
and the weight functions we can find a link
between the polynomials $(P^{\alpha,q}_n)$ and
$GG-$polynomials, which was not possible to do
with Laguerre-Freud equation \cite{kn:atia1} nor
with the explicit representation of
$P^{\alpha,q}_n(x)$ \cite{kn:atia2}.
\end{rem}

\begin{prop}
The monic polynomial sequence $(P_{n}^{\alpha,
q}(x))$  fulfills for $n\ge 0$ the following TTRR:
\begin{equation}\label{artic15}
P_{n+2}^{\alpha,q}(x)=(x-\beta_{n+1}^{\alpha,q})
P_{n+1}^{\alpha,q}(x)-\gamma_{n+1}^{\alpha,q}
P_{n}^{\alpha,q}(x),
\end{equation}
with initial conditions $P_{0}^{\alpha,q}(x)=1$,
$P_{1}^{\alpha,q}(x)=x-\beta_{0}^{\alpha,q}$;
where
\begin{equation}\label{artic16}
\beta_{n}^{\alpha,q}=(-1)^{n+1},
\end{equation}
\begin{equation}\label{artic17}
\gamma_{2n}^{\alpha,q}=-2\frac{n(2n+2q+1)}
{(4n+2\alpha+2q+1)(4n+2\alpha+2q+3)},
\end{equation}
\begin{equation}\label{artic18}
\gamma_{2n+1}^{\alpha,q}=-2\frac{(n+\alpha+1)(2n+2
\alpha+2q+3)}{(4n+2\alpha+2q+3)(4n+2\alpha+2q+5)}.
\end{equation}
\end{prop}
\begin{proof}
For all $n\geq0$ we have
\begin{equation*}
\beta_{n}^{\alpha,q}=\frac{\int_{-1}^{1}
\big(x^{2q+1}(1-x^{2})^{\alpha}(1-x)\big)
x(P_{n}^{\alpha,q}(x))^{2}dx}{\int_{-1}^{1}
\big(x^{2q+1}(1-x^{2})^{\alpha}(1-x)\big)
(P_{n}^{\alpha,q}(x))^{2}dx}.
\end{equation*}
So, setting $2n\leftarrow n$ one gets
\begin{equation*}
\beta_{2n}^{\alpha,q}=
\frac{\int_{-1}^{1}x^{2q+2}(1-x^2)^{\alpha}
(1-x)(P_{2n}^{\alpha,q}(x))^{2}dx}{\int_{-1}^{1}
x^{2q+1}(1-x^2)^{\alpha}(1-x)(P_{2n}^{\alpha,q}
(x))^{2}dx},
\end{equation*}
since $P_{2n}^{\alpha,q}$ is even we obtain
\begin{equation*}
\beta_{2n}^{\alpha,q}=\frac{\int_{-1}^{1}x^{2q+2}
(1-x^2)^{\alpha}(P_{2n}^{\alpha,q}(x))^{2}(x)dx}
{-\int_{-1}^{1}x^{2q+2}
(1-x^2)^{\alpha}(P_{2n}^{\alpha,q}(x))^{2}dx}=-1.
\end{equation*}
The odd case is completely analogous and it
will be omitted.

In order to obtain $\gamma_{n+1}^{\alpha,q}$ we
use \eqref{artic15} together with
$\beta_{n}^{\alpha,q}=(-1)^{n+1}$, i.e. since
for $n\ge 0$ we get
\begin{equation*}
P_{n+2}^{\alpha,q}(x)=
(x-(-1)^{n})P_{n+1}^{\alpha,q}(x)-
\gamma_{n+1}^{\alpha,q}P_{n}^{\alpha,q}(x),
\end{equation*}
then, taking into account Eq. \eqref{p2np1gg2n2},
we deduce
\begin{equation*}
\gamma_{2n}^{\alpha,q}=\frac{P_{2n}^{\alpha,q}(x)
-P_{2n}^{\alpha+1,q}(x)}{P_{2n-2}^{\alpha+1,q}(x)},
\end{equation*}
\begin{equation*}
\gamma_{2n+1}^{\alpha,q}=\frac{(x^{2}-1)
P_{2n}^{\alpha+1,q}(x)-P_{2n+2}^{\alpha,q}(x)}
{P_{2n}^{\alpha,q}(x)}.
\end{equation*}
To compute $\gamma_{2n}^{\alpha,q}$ we use
\eqref{p2nhyper2} and \eqref{contig4} with
$a=n+q+\alpha+\frac 32$, $b=-n$, $c=q+\frac 32$,
and $t=x^{2}$, obtaining
\begin{equation*}
\gamma_{2n}^{\alpha,q}
=\frac{(-1)^{n}\frac{(c)_{n}}
{(a)_{n}}F-(-1)^{n}\frac{(c)_{n}}{(a+1)_{n}}F(a+1)}
{(-1)^{n-1}\frac{(c)_{n-1}}{(a)_{n-1}}F(b+1)},
\end{equation*}
and since
$(\lambda+1)_{k}=\frac{\lambda+k}{\lambda}
(\lambda)_{k}$ we get
\begin{equation*}
\gamma_{2n}^{\alpha,q}=-\frac{(c)_{n}(a)_{n-1}}
{(c)_{n-1}(a)_{n}}\frac{F-\frac{a}{a+n}F(a+1)}
{F(b+1)}
\end{equation*}
\begin{equation*}
=-\frac{(c)_{n}(a)_{n-1}}{(a-b)(c)_{n-1}(a)_{n}}
\frac{(a-b)F-aF(a+1)}{F(b+1)}=b\frac{(c)_{n}
(a)_{n-1}}{(a-b)(c)_{n-1}(a)_{n}},
\end{equation*}
so
\begin{equation*}
\gamma_{2n}^{\alpha,q}=b\frac{(c)_{n}(a)_{n-1}}
{(a-b)(c)_{n-1}(a)_{n}}
\end{equation*}
thus
\begin{align*}
\gamma_{2n}^{\alpha,q}=&-n\frac{(q+\frac 32)_{n}
(n+q+\alpha+\frac 32)_{n-1}}
{(2n+q+\alpha+\frac 32)
(q+\frac 32)_{n-1}(n+q+\alpha+\frac 32)_{n}}\\
&=-n\frac{(q+\frac 32+n-1)}{(2n+q+\alpha+\frac
32)(2n+q+\alpha+\frac 32-1)}\\
&=-2\frac{n(2n+2q+1)}{(4n+2q+2\alpha+1)(4n+2q+2
\alpha+3)}.
\end{align*}
To compute $\gamma_{2n+1}$ we use \eqref{p2nhyper2}
and \eqref{contig5} with $a=n+q+\alpha+\frac 52$,
$b=-n$, $c=q+\frac 32$, and $t=x^{2}$, obtaining
\begin{equation*}
\gamma_{2n+1}^{\alpha,q}=\frac{(-1)^{n}
\frac{(c)_{n}}{(a)_{n}}(t-1)F-(-1)^{n+1}
\frac{(c)_{n+1}}{(a)_{n+1}}F(b-1)}
{(-1)^{n}\frac{(c)_{n}}{(a-1)_{n}}F(a-1)}.
\end{equation*}
Since $(c)_{n+1}=(c+n)(c)_{n}$ and
$(a)_{n+1}=(a+n)(a)_{n}$ we have
\begin{equation*}
\gamma_{2n+1}^{\alpha,q}=\frac{(a-1)_{n}}
{(a+n)(a)_{n}}\frac{(t-1)(a-b)F+(c-b)F(b-1)}
{F(a-1)}=(c-a)\frac{(a-1)_{n}}{(a+n)(a)_{n}},
\end{equation*}
thus
\begin{align*}
\gamma_{2n+1}^{\alpha,q}=&(q+\mbox{$\frac 32$}-(n+q+
\alpha+\mbox{$\frac 52$}))\frac{(n+q+\alpha+
\frac 52-1)_{n}}{(2n+q+\alpha+\frac 52)(n+q+
\alpha+\frac 52)_{n}}\\ & =-2\frac{(n+\alpha+1)
(2n+2\alpha+2q+3)}{(4n+2\alpha+2q+3)(4n+2\alpha
+2q+5)}.
\end{align*}
\end{proof}
\section{Zeros of $(P_{n}^{\alpha,q})$}
Using \eqref{relationpmu}, \eqref{p2ngg2n} and
\eqref{p2np1gg2n} we can state the following
result:
\begin{thm} The following statements hold:
\begin{itemize}
\item All the zeros of  $P_{2n}^{\alpha,q}(x)$ are
real.
\item The Perron's zero is $-1$.
\item The zeros of  $P_{2n}^{\alpha,q}(x)$ and
the zeros of $P_{2n+1}^{\alpha,q}(x)$ do not
interlace.
\end{itemize}
\end{thm}
\begin{proof}
By \eqref{relationpmu} and  \eqref{p2ngg2n}, the
first two statements follow.
To prove the third one, it is sufficient to see
that the zeros of $GG_{2n}^{\alpha,\mu}(x)$ and
$GG_{2n}^{\alpha+1,\mu}(x)$ do not interlace.
But, for all $x\in[-1,1]$ we know that
$GG_{2n}^{\alpha,\mu}(-x)=GG_{2n}^{\alpha,\mu}(x)$
thus $GG_{2n}^{\alpha,2q+2}$ and
$GG_{2n}^{\alpha+1,2q+2}$ have exactly $n$ zeros
in $]0,1[$.
\\
Let $(x_{2n,k}^{\alpha,2q+2})_{1\leq k\leq 2n}$ and
$(x_{2n,k}^{\alpha+1,2q+2})_{1\leq k\leq 2n}$ be
the zeros of $GG_{2n}^{\alpha,2q+2}$ and
$GG_{2n}^{\alpha+1,2q+2}$ respectively in
increasing order.
Then, between $-x_{0}$ and $x_{0}$ --that are
consecutive zeros of $GG_{2n}^{\alpha,2q+2}(x)$
(resp. of $GG_{2n}^{\alpha+1,2q+2}$)-- we can not
find a zero of $GG_{2n}^{\alpha+1,2q+2}$ (resp. of
$GG_{2n}^{\alpha,2q+2}$).

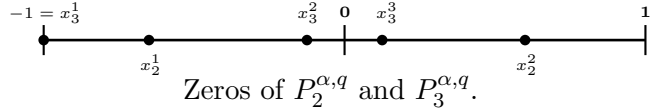
\begin{center}
\setlength{\unitlength}{0.00045833in}
\begingroup\makeatletter\ifx\SetFigFont\undefined%
\gdef\SetFigFont#1#2#3#4#5{%
  \reset@font\fontsize{#1}{#2pt}%
  \fontfamily{#3}\fontseries{#4}\fontshape{#5}%
  \selectfont}%
\fi\endgroup%
{\renewcommand{\dashlinestretch}{30}
\begin{picture}(10000,0)(0,0)
\psline[linewidth=1pt,showpoints=true](2,0)(3.4,0)(5.5,0)(6.5,0)(8.4,0)
\psline[linewidth=1pt](8.4,0)(10,0)
\psline(2,-0.2)(2,0.2)
\psline(10,-0.2)(10,0.2)
\psline(6,-0.2)(6,0.2)
\put(1320,200){\makebox(0,0)[lb]
{{\SetFigFont{6}{7.2}{\rmdefault}
{\mddefault}{\updefault}$-1=x_3^1$}}}
\put(2820,-400){\makebox(0,0)[lb]
{{\SetFigFont{6}{7.2}{\rmdefault}
{\mddefault}{\updefault}$x_2^1$}}}
\put(4620,200){\makebox(0,0)[lb]
{{\SetFigFont{6}{7.2}{\rmdefault}
{\mddefault}{\updefault}$x_3^2$}}}
\put(5100,270){\makebox(0,0)[lb]
{{\SetFigFont{6}{7.2}{\rmdefault}
{\mddefault}{\updefault}{\bf 0}}}}
\put(5520,200){\makebox(0,0)[lb]
{{\SetFigFont{6}{7.2}{\rmdefault}
{\mddefault}{\updefault}$x_3^3$}}}
\put(7120,-400){\makebox(0,0)[lb]
{{\SetFigFont{6}{7.2}{\rmdefault}
{\mddefault}{\updefault}$x_2^2$}}}
\put(8540,270){\makebox(0,0)[lb]
{{\SetFigFont{6}{7.2}{\rmdefault}
{\mddefault}{\updefault}{\bf 1}}}}
\end{picture}
\\[3mm]
Zeros of $P_{2}^{\alpha,q}$ and $P_{3}^{\alpha,q}$.
}
\end{center}
\end{proof}

The following result allow us to obtain even more
information regarding the zeros of
$(P_{2n}^{\alpha,q}(x))$:
\begin{prop}
For $n\ge 0$, it holds:
\begin{equation}\label{zeros1}
\frac{d}{dx}P_{2n}^{\alpha,q}(x)
=2nGG_{2n-1}^{\alpha+1,2q+2}(x).
\end{equation}
Taking into account \eqref{p2np1gg2n2} and
\eqref{relationpmu}, we can write the last
equality as
\begin{equation}\label{zeros2}
\frac{d}{dx}P_{2n}^{\alpha,q}(x)
=2nxP_{2n-2}^{\alpha+1,q+1}(x).
\end{equation}
We also have
\begin{equation}\label{zeros3}
\frac{d}{dx}P_{2n+1}^{\alpha,q}(x)=P_{2n}^{\alpha
+1,q}(x)+2nxP_{2n-1}^{\alpha+1,q+1}(x).
\end{equation}
\end{prop}
\begin{proof}
Let us start proving \eqref{zeros1}.
To do this it is enough to prove
 \begin{equation*}
\int_{-1}^{1}x^{2q+2}(1-x^2)^{\alpha+1}x^{k}
\Big(P_{2n}^{\alpha,q}(x)\Big)'dx=0,\qquad
0\leq k\leq 2n-2.
\end{equation*}
But, integrating by parts once one gets
\begin{equation*}
\int_{-1}^{1}x^{2q+k+2}(1-x^2)^{\alpha+1}
\Big(P_{2n}^{\alpha,q}(x)\Big)'dx=-\int_{-1}^{1}
\Big((2q+k+2)x^{2q+k+1}(1-x^2)^{\alpha+1}
\end{equation*}
\begin{equation*}
-2(\alpha+1)x^{2q+k+3}(1-x^2)^{\alpha}\Big)
P_{2n}^{\alpha,q}(x)dx
\end{equation*}
\begin{equation*}
=\int_{-1}^{1}x^{2q+1}(1-x^2)^{\alpha}
\Big(2(\alpha+1)x^{2}-(2q+k+2)(1-x^2)
\Big)x^{k}P_{2n}^{\alpha,q}(x)dx.
\end{equation*}
If $k=2p,\ 0\leq p\leq n-1$ then
\begin{equation*} \int_{-1}^{1}x^{2q+1}(1-x^2)^{
\alpha}\Big((2\alpha+2q+2p+4)x^{2}-(2q+2p+2)\Big)
x^{2p}P_{2n}^{\alpha,q}(x)dx=0,
\end{equation*}
since we have an odd function in $[-1,1]$.
\\
 If $k=2p+1,\ 0\leq p\leq n-2$ then
\begin{equation*}
\int_{-1}^{1}x^{2q+1}(1-x^2)^{\alpha}\Big(2
(\alpha+1)x^{2}-(2q+2p+3)(1-x^2)
\Big)x^{2p+1}P_{2n}^{\alpha,q}(x)dx
\end{equation*}
\begin{equation*}
=2(\alpha+1)\int_{-1}^{1}x^{2q+1}(1-x^2)^{\alpha}
x^{2p+3}P_{2n}^{\alpha,q}(x)dx
\end{equation*}
\begin{equation*}
-(2q+2p+3)\int_{-1}^{1}x^{2q+1}(1-x^2)^{\alpha}
(1-x)\Big(x^{2p+1}(1+x)\Big)P_{2n}^{\alpha,q}(x)dx,
\end{equation*}
where the second integral vanishes since the
property of orthogonality holds and $2p+2<2n$,
and the first integral can be written as
\begin{equation*}
2(\alpha+1)\int_{-1}^{1}x^{2q+1}(1-x^2)^{\alpha}
(1-x)x^{2p+3}P_{2n}^{\alpha,q}(x)dx,\qquad
0\leq p\leq n-2,
\end{equation*}
that also vanishes because $2p+3\leq 2n-1$.

Therefore
\begin{equation*}
\int_{-1}^{1}x^{2q+2}(1-x^2)^{\alpha+1}x^{k}
\Big(P_{2n}^{\alpha,q}(x)\Big)'dx=0,\qquad
0\leq k\leq 2n-2.
\end{equation*}
But $\deg\Big(P_{2n}^{\alpha,q}(x)\Big)'=2n-1$, so
\begin{equation*}
\frac{d}{dx}P_{2n}^{\alpha,q}(x)=2nGG_{2n-1}^{
\alpha+1,2q+2}(x),\ n\geq1.
\end{equation*}
By using Eq. \eqref{zeros1} and recalling identity
\eqref{relationpmu}, then \eqref{zeros2} holds.
To prove \eqref{zeros3}, we need to replace
$P_{2n+1}^{\alpha,q}$ by $(1+x)P_{2n}^{\alpha+1,q}$
getting
\begin{equation*}
\frac {d}{dx} P_{2n+1}^{\alpha,q}(x)=\Big((1+x)
P_{2n}^{\alpha+1,q}(x)\Big)'=(1+x)\Big(P_{2n}^{
\alpha+1,q}(x)\Big)'+P_{2n}^{\alpha+1,q}(x),
\end{equation*}
taking into account \eqref{zeros2} we have
\begin{equation*}
\Big(P_{2n+1}^{\alpha,q}(x)\Big)'=
(1+x)\Big(2nxP_{2n-2}^{\alpha+2,q+1}
(x)\Big)+P_{2n}^{\alpha+1,q}(x),
\end{equation*}
and using $P_{2n-1}^{\alpha+1,q+1}(x)
=(1+x)P_{2n-2}^{\alpha+2,q+1}(x)$ the identity
holds.
\end{proof}

Note that these relations help us to obtain
more information related to the zeros of both,
$(P_{n}^{\alpha,q})$ and $(P_{n}^{\alpha+i,q+j})$,
for any $i,j\geq0$.
The following result show us how relevant is that
relation:
\begin{prop}
If we denote by $x_{n,n}^{\alpha+i,q+j}$ the
largest zero of $P_{n}^{\alpha+i,q+j}$, then
\begin{equation}\label{relazero}
 x_{2n-2j,2n-2j}^{\alpha+j,q+j}<x_{2n-2i,2n-2i}^{
 \alpha+i,q+i},\qquad 0\leq i<j\leq n-1.
 \end{equation}
\end{prop}
\begin{proof}
To prove \eqref{relazero} we need to use
\eqref{zeros2}:
$$
\frac{d}{dx}P_{2n}^{\alpha,q}(x)=2nxP_{2n-2}^{
\alpha+1,q+1}(x),\qquad n\geq0.
$$\\
We know that $P_{2n}^{\alpha,q}(x)$ and
$P_{2n-2}^{\alpha+1,q+1}(x)$ have $n$ and
$(n-1)$ zeros in $]0,1[$ respectively and between
two consecutive zeros of $P_{2n}^{\alpha,q}(x)$
we find, exactly, one zero of $P_{2n-2}^{\alpha+1,
q+1}(x)$ then the largest zero of $P_{2n-2}^{
\alpha+1,q+1}(x)$ is located between two zeros of
$P_{2n}^{\alpha,q}(x)$ thus
$$
x_{2n-2,2n-2}^{\alpha+1,q+1}<x_{2n,2n}^{\alpha,q},
$$
an analog idea leads to
$$
x_{2n-4,2n-4}^{\alpha+2,q+2}<x_{2n-2,2n-2}^{
\alpha+1,q+1},$$
and so on.
Then, we can write
\begin{equation*}
\dots<x_{2n-2j,2n-2j}^{\alpha+j,q+j}<\dots<x_{2n-2i,2n-2i}^{\alpha+i,q+i}
<\dots<x_{2n-2,2n-2}^{\alpha+1,q+1}<x_{2n,2n}^{\alpha,q}.
\end{equation*}
Hence the result holds.
\end{proof}
\begin{rem}
Using Proposition 3.4. and the relation
\begin{equation*}
x_{2n+1,m}^{\alpha+k,q+l}=x_{2n,m-1}^{\alpha+k+1,
q+l};\qquad k, l\in\mathbb{N}, 2\leq m\leq 2n+1,
\end{equation*}
with $x_{2n+1,1}^{\alpha+k,q+l}=-1$, we obtain
\begin{equation*}
x_{2n-2j+1,2n-2j+1}^{\alpha+j-1,q+j}<x_{2n-2i+1,
2n-2i+1}^{\alpha+i-1,q+i},\qquad 0\leq i<j\leq n.
\end{equation*}
\end{rem}

\noindent{\bf Acknowledgment.} The authors want
to thank Prof. F. Marcellan and Prof. H. Stahl
for their discussions and remarks which helped
us to improve the representation of this paper.

\end{document}